\documentclass{conm-p-l}
\usepackage{amssymb}

\usepackage{amsfonts}
\usepackage{eucal}
\usepackage{upgreek}
\begin{document}

\font\eightrm=cmr8
\font\eightit=cmti8
\font\eighttt=cmtt8
\font\tensans=cmss10
\def\emp{\text{\tensans\O}}
\def\tci
{\hbox{\hskip1.8pt$\rightarrow$\hskip-11.5pt$^{^{C^\infty}}$\hskip-1.3pt}}
\def\nft
{\hbox{$n$\hskip3pt$\equiv$\hskip4pt$5$\hskip4.4pt$($mod\hskip2pt$3)$}}
\def\bbR{\mathrm{I\!R}}
\def\rto{\bbR\hskip-.5pt^2}
\def\rtr{\bbR\hskip-.7pt^3}
\def\rfo{\bbR\hskip-.7pt^4}
\def\rn{\bbR^{\hskip-.6ptn}}
\def\mr{\bbR^{\hskip-.6ptm}}
\def\bbZ{\mathsf{Z\hskip-4ptZ}}
\def\bbRP{\text{\bf R}\text{\rm P}}
\def\bbC{{\mathchoice {\setbox0=\hbox{$\displaystyle\rm C$}\hbox{\hbox
to0pt{\kern0.4\wd0\vrule height0.9\ht0\hss}\box0}}
{\setbox0=\hbox{$\textstyle\rm C$}\hbox{\hbox
to0pt{\kern0.4\wd0\vrule height0.9\ht0\hss}\box0}}
{\setbox0=\hbox{$\scriptstyle\rm C$}\hbox{\hbox
to0pt{\kern0.4\wd0\vrule height0.9\ht0\hss}\box0}}
{\setbox0=\hbox{$\scriptscriptstyle\rm C$}\hbox{\hbox
to0pt{\kern0.4\wd0\vrule height0.9\ht0\hss}\box0}}}}
\def\bbQ{{\mathchoice {\setbox0=\hbox{$\displaystyle\rm Q$}\hbox{\raise
0.15\ht0\hbox to0pt{\kern0.4\wd0\vrule height0.8\ht0\hss}\box0}}
{\setbox0=\hbox{$\textstyle\rm Q$}\hbox{\raise
0.15\ht0\hbox to0pt{\kern0.4\wd0\vrule height0.8\ht0\hss}\box0}}
{\setbox0=\hbox{$\scriptstyle\rm Q$}\hbox{\raise
0.15\ht0\hbox to0pt{\kern0.4\wd0\vrule height0.7\ht0\hss}\box0}}
{\setbox0=\hbox{$\scriptscriptstyle\rm Q$}\hbox{\raise
0.15\ht0\hbox to0pt{\kern0.4\wd0\vrule height0.7\ht0\hss}\box0}}}}
\def\bbQ{{\mathchoice {\setbox0=\hbox{$\displaystyle\rm Q$}\hbox{\raise
0.15\ht0\hbox to0pt{\kern0.4\wd0\vrule height0.8\ht0\hss}\box0}}
{\setbox0=\hbox{$\textstyle\rm Q$}\hbox{\raise
0.15\ht0\hbox to0pt{\kern0.4\wd0\vrule height0.8\ht0\hss}\box0}}
{\setbox0=\hbox{$\scriptstyle\rm Q$}\hbox{\raise
0.15\ht0\hbox to0pt{\kern0.4\wd0\vrule height0.7\ht0\hss}\box0}}
{\setbox0=\hbox{$\scriptscriptstyle\rm Q$}\hbox{\raise
0.15\ht0\hbox to0pt{\kern0.4\wd0\vrule height0.7\ht0\hss}\box0}}}}
\def\aff{\mathrm{A\hn f\hh f}\hs}
\def\cx{C\hskip-2pt_x\w}
\def\cy{C\hskip-2pt_y\w}
\def\cz{C\hskip-2pt_z\w}
\def\hyp{\hskip.5pt\vbox
{\hbox{\vrule width3ptheight0.5ptdepth0pt}\vskip2.2pt}\hskip.5pt}
\def\er{r}
\def\es{s}
\def\df{d\hskip-.8ptf}
\def\dz{\mathcal{D}}
\def\dzp{\dz^\perp}
\def\fv{\mathcal{F}}
\def\gr{\mathcal{G}}
\def\fvp{\fv_{\nrmh p}}
\def\wv{\mathcal{W}}
\def\vt{\mathcal{P}}
\def\tv{\mathcal{T}}
\def\vr{\mathcal{V}}
\def\xs{J}
\def\xl{S}
\def\cs{\mathcal{B}}
\def\zy{\mathcal{Z}}
\def\vtx{\vt_{\nh x}}
\def\fh{f}
\def\g{\mathtt{g}}
\def\rc{\theta}
\def\jm{\mathcal{I}}
\def\ke{\mathcal{K}}
\def\xc{\mathcal{X}_c}
\def\lz{\mathcal{L}}
\def\dla{\mathcal{D}_{\hskip-2ptL}^*}
\def\Lie{\pounds}
\def\lv{\Lie\hskip-1.2pt_v\w}
\def\lo{\lz_0}
\def\xe{\mathcal{E}}
\def\eo{\xe_0}
\def\lsq{\mathsf{[}}
\def\rsq{\mathsf{]}}
\def\hga{\hskip2.3pt\widehat{\hskip-2.3pt\gamma\hskip-2pt}\hskip2pt}
\def\hm{\hskip1.9pt\widehat{\hskip-1.9ptM\hskip-.2pt}\hskip.2pt}
\def\hg{\hskip.9pt\widehat{\hskip-.9pt\g\hskip-.9pt}\hskip.9pt}
\def\hna{\hskip.2pt\widehat{\hskip-.2pt\nabla\hskip-1.6pt}\hskip1.6pt}
\def\hdz{\hskip.9pt\widehat{\hskip-.9pt\dz\hskip-.9pt}\hskip.9pt}
\def\hdp{\hskip.9pt\widehat{\hskip-.9pt\dz\hskip-.9pt}\hskip.9pt^\perp}
\def\hmt{\hskip1.9pt\widehat{\hskip-1.9ptM\hskip-.5pt}_t}
\def\hmz{\hskip1.9pt\widehat{\hskip-1.9ptM\hskip-.5pt}_0}
\def\hmp{\hskip1.9pt\widehat{\hskip-1.9ptM\hskip-.5pt}_p}
\def\hk{\hskip1.5pt\widehat{\hskip-1.5ptK\hskip-.5pt}\hskip.5pt}
\def\hq{\hskip1.5pt\widehat{\hskip-1.5ptQ\hskip-.5pt}\hskip.5pt}
\def\txm{{T\hskip-3.5pt_x\w M}}
\def\tyhm{{T\hskip-3.5pt_y\w\hm}}
\def\q{q}
\def\bq{\hat q}
\def\p{p}
\def\w{^{\phantom i}}
\def\x{v}
\def\y{y}
\def\vp{{\tau\hskip-4.55pt\iota\hskip.6pt}}
\def\evp{{\tau\hskip-3.55pt\iota\hskip.6pt}}
\def\vd{\vt\hh'}
\def\vdx{\vd{}\hskip-4.5pt_x}
\def\bz{b\hh}
\def\fe{F}
\def\fy{\phi}
\def\vl{\Lambda}
\def\hy{\mathcal{V}}
\def\vh{h}
\def\bc{C}
\def\mv{V}
\def\vo{V_{\nnh0}}
\def\ao{A_0}
\def\bo{B_0}
\def\uv{\mathcal{U}}
\def\sv{\mathcal{S}}
\def\svp{\sv_p}
\def\xv{\mathcal{X}}
\def\xvp{\xv_p}
\def\yv{\mathcal{Y}}
\def\yvp{\yv_p}
\def\zv{\mathcal{Z}}
\def\zvp{\zv_p}
\def\cv{\mathcal{C}}
\def\dy{\mathcal{D}}
\def\nv{\mathcal{N}}
\def\iv{\mathcal{I}}
\def\gkp{\Sigma}
\def\ret{\sigma}
\def\taw{\uptau}%{{\tau\hskip-4.55pt\iota\hskip.6pt}}%
\def\hs{\hskip.7pt}
\def\hh{\hskip.4pt}
\def\hn{\hskip-.4pt}
\def\nh{\hskip-.7pt}
\def\nnh{\hskip-1pt}
\def\hrz{^{\hskip.5pt\text{\rm hrz}}}
\def\vrt{^{\hskip.2pt\text{\rm vrt}}}
\def\vt{\varTheta}
\def\mtr{\Theta}
\def\op{\varTheta}
\def\vg{\varGamma}
\def\my{\mu}
\def\ny{\nu}
\def\gy{\lambda}
\def\lp{\lambda}
\def\ax{\alpha}
\def\lf{\widetilde{\lp}}
\def\bx{\beta}
\def\ay{a}
\def\by{b}
\def\gp{\mathrm{G}}
\def\hp{\mathrm{H}}
\def\kp{\mathrm{K}}
\def\gm{\gamma}
\def\Gm{\Gamma}
\def\Lm{\Lambda}
\def\Dt{\Delta}
\def\dg{\Delta}
\def\sj{\sigma}
\def\lg{\langle}
\def\rg{\rangle}
\def\lr{\langle\hh\cdot\hs,\hn\cdot\hh\rangle}
\def\vs{vector space}
\def\rvs{real vector space}
\def\vf{vector field}
\def\tf{tensor field}
\def\tvn{the vertical distribution}
\def\dn{distribution}
\def\pt{point}
\def\tc{tor\-sion\-free connection}
\def\ea{equi\-af\-fine}
\def\rt{Ric\-ci tensor}
\def\pde{partial differential equation}
\def\pf{projectively flat}
\def\pfs{projectively flat surface}
\def\pfc{projectively flat connection}
\def\pftc{projectively flat tor\-sion\-free connection}
\def\su{surface}
\def\sco{simply connected}
\def\psr{pseu\-\hbox{do\hs-}Riem\-ann\-i\-an}
\def\inv{-in\-var\-i\-ant}
\def\trinv{trans\-la\-tion\inv}
\def\feo{dif\-feo\-mor\-phism}
\def\feic{dif\-feo\-mor\-phic}
\def\feicly{dif\-feo\-mor\-phi\-cal\-ly}
\def\Feicly{Dif-feo\-mor\-phi\-cal\-ly}
\def\diml{-di\-men\-sion\-al}
\def\prl{-par\-al\-lel}
\def\skc{skew-sym\-met\-ric}
\def\sky{skew-sym\-me\-try}
\def\Sky{Skew-sym\-me\-try}
\def\dbly{-dif\-fer\-en\-ti\-a\-bly}
\def\cf{con\-for\-mal\-ly flat}
\def\ls{locally symmetric}
\def\ecs{essentially con\-for\-mal\-ly symmetric}
\def\rr{Ric\-ci-re\-cur\-rent}
\def\kf{Killing field}
\def\om{\omega}
\def\vol{\varOmega}
\def\dv{\delta}
\def\ve{\varepsilon}
\def\zt{\zeta}
\def\kx{\kappa}
\def\mf{manifold}
\def\mfd{-man\-i\-fold}
\def\bmf{base manifold}
\def\bd{bundle}
\def\tbd{tangent bundle}
\def\ctb{cotangent bundle}
\def\bp{bundle projection}
\def\prc{pseu\-\hbox{do\hs-}Riem\-ann\-i\-an metric}
\def\prd{pseu\-\hbox{do\hs-}Riem\-ann\-i\-an manifold}
\def\Prd{pseu\-\hbox{do\hs-}Riem\-ann\-i\-an manifold}
\def\npd{null parallel distribution}
\def\pj{-pro\-ject\-a\-ble}
\def\pd{-pro\-ject\-ed}
\def\lcc{Le\-vi-Ci\-vi\-ta connection}
\def\vb{vector bundle}
\def\vbm{vec\-tor-bun\-dle morphism}
\def\kerd{\text{\rm Ker}\hskip2.7ptd}
\def\ro{\rho}
\def\sy{\sigma}
\def\ts{total space}
\def\pmb{\pi}

\newtheorem{theorem}{Theorem}[section] 
\newtheorem{proposition}[theorem]{Proposition} 
\newtheorem{lemma}[theorem]{Lemma} 
\newtheorem{corollary}[theorem]{Corollary} 
  
\theoremstyle{definition} 
  
\newtheorem{defn}[theorem]{Definition} 
\newtheorem{notation}[theorem]{Notation} 
\newtheorem{example}[theorem]{Example} 
\newtheorem{conj}[theorem]{Conjecture} 
\newtheorem{prob}[theorem]{Problem} 
  
\theoremstyle{remark} 
  
\newtheorem{remark}[theorem]{Remark}

\title[Parallel Weyl tensor]{Compact locally homogeneous manifolds\\ 
with parallel Weyl tensor}
\author[A. Derdzinski]{Andrzej Derdzinski} 
\address{Department of Mathematics, The Ohio State University, 
Columbus, OH 43210} 
\email{andrzej@math.ohio-state.edu} 
\author[I.\ Terek]{Ivo Terek} 
\address{Department of Mathematics, The Ohio State University, 
Columbus, OH 43210} 
\email{terekcouto.1@osu.edu} 
\def\leftmark{A.\ Derdzinski \&\ I.\ Terek}
\def\rightmark{Parallel Weyl tensor}

\begin{abstract}
We construct new examples of compact ECS manifolds, that is, of 
pseu\-do\hs-Riem\-ann\-i\-an manifolds with parallel Weyl tensor that are
neither con\-for\-mal\-ly flat nor locally symmetric. Every ECS manifold has
rank 1 or 2, the rank being the dimension of a distinguished null parallel
distribution discovered by Ol\-szak. 
Previously known examples of compact ECS manifolds, in every dimension  
greater than 4, were all of rank 1, geodesically complete, and 
none of them locally homogeneous. 
By contrast, our new examples -- all of them geodesically incomplete --
realize all odd dimensions starting from 5 and are this time of rank 2, as
well as locally homogeneous.
\end{abstract}

\maketitle

\setcounter{section}{0}
\setcounter{theorem}{0}
\renewcommand{\theequation}{\arabic{section}.\arabic{equation}}
\section*{Introduction}%\label{in}
\setcounter{equation}{0}
By an {\it ECS manifold\/} \cite{derdzinski-roter-07} -- short for 
`essentially con\-for\-mal\-ly symmetric' -- one means a 
pseu\-\hbox{do\hskip.7pt-}Riem\-ann\-i\-an manifold 
of dimension $\,n\ge4\,$ having nonzero parallel Weyl tensor $\,W\nnh$, and
not being locally symmetric. Its {\it rank\/} $\,d\in\{1,2\}\,$ is the 
dimension of its {\it Ol\-szak distribution\/}
\cite{olszak}, \cite[p.\ 119]{derdzinski-roter-09}, the null parallel 
distribution $\,\dz$, the sections of which are the vector fields
corresponding via the metric to $\,1$-forms $\,\xi\,$ such that 
$\,\xi\wedge[W(v,v'\nh,\,\cdot\,,\,\cdot\,)]=0\,$ for all 
vector fields $\,v,v'\nh$. (The term `con\-for\-mal\-ly symmetric' should not
be misconstrued as referring to con\-for\-mal geometry.)

ECS manifolds are of obvious interest
\cite{cahen-kerbrat,hotlos,mantica-suh,schliebner,
deszcz-glogowska-hotlos-zafindratafa,
deszcz-glogowska-hotlos-petrovic-torgasev-zafindratafa} 
due to naturality and simplicity of
the condition $\,\nabla W\hskip-1.7pt=0$. 
Roter proved the existence of ECS manifolds 
\cite[Corol\-lary~3]{roter} in all dimensions $\,n\ge4\,$ and showed that
their metrics are necessarily indefinite
\cite[Theorem~2]{derdzinski-roter-77}. Locally homogeneous ECS manifolds of
either rank exist \cite{derdzinski-78} for all $\,n\ge4$. The 
local structure of ECS manifolds has been completely described
\cite{derdzinski-roter-09}.

Examples of {\it compact rank-one\/} ECS manifolds are known 
\cite{derdzinski-roter-10,derdzinski-terek-ne} in every dimension $\,n\ge5$. 
They are geodesically complete and not locally homogeneous, which 
%The examples in \cite{derdzinski-roter-10,derdzinski-terek-ne} thus
raises three obvious questions: Can a compact ECS manifold have rank two, or 
be locally homogeneous, or geodesically incomplete?

This paper answers all three in the affirmative, for every {\it odd\/}
dimension $\,n\ge5$.

Just like in \cite{derdzinski-roter-10,derdzinski-terek-ne}, our examples are 
dif\-feo\-mor\-phic to nontrivial torus bundles over the circle, and arise as
quotients of certain explicitly described simply connected ``model'' 
manifolds $\,\hm\,$ under free and properly dis\-con\-tin\-u\-ous actions on
$\,\hm$ of suitable groups $\,\Gm\,$ of isometries. However, selecting such
objects involves two aspects, analytical for $\,\hm\,$ (the existence of a
specific function $\,f\,$ of a real variable) and combinatorial for 
$\,\Gm\nh$, and it is here that our approach fundamentally differs from 
\cite{derdzinski-roter-10} and \cite{derdzinski-terek-ne}. Whereas in those
two papers the combinatorial part was trivial, and finding $\,f\,$ required
extensive work -- a messy explicit construction in
\cite{derdzinski-roter-10}, only good for dimensions $\,n\,$ congruent to
$\,5\,$ modulo $\,3$, and a deformation argument applied to uninteresting
constant functions in \cite{derdzinski-terek-ne} -- the situation here is the
exact opposite: $\,f\,$ comes from the very simple formula (\ref{exa}), while 
the groups $\,\Gm\hs$ arise via combinatorial structures ($\bbZ$-spec\-tral
systems), the existence of which we can only establish, with some effort, in
Theorem~\ref{zspsy}, for odd dimensions $\,n$.

Every $\,\bbZ$-spec\-tral system gives rise to a free Abel\-i\-an group
$\,\Sigma\,$ of isometries in each model manifold of a suitable type,
associated with a narrow class of choices of the function $\,f\nh$, so that
$\,\Sigma\,$ satisfies conditions (\ref{ace}), which in turn allows us 
to extend $\,\Sigma\,$ to the required group $\,\Gm\nh$, leading
to a compact quotient manifold. See Theorem~\ref{ccdro}. (Our argument used
to derive Theorem~\ref{ccdro} from (\ref{ace}) is a modified version of those
in \cite{derdzinski-roter-10} and \cite{derdzinski-terek-ne}.) 
One such choice of $\,f\nh$, namely, (\ref{exa}), makes the resulting compact 
rank-two ECS manifolds locally homogeneous (Theorem~\ref{lchom}).
They are also all incomplete, for rather obvious reasons (Remark~\ref{lhinc}).

The preceding sentence leads to a further question: For a compact ECS
manifold, can one have 
incompleteness without local homogeneity? We answer it in the affirmative -- 
with any $\,f\,$ given by (\ref{exa}), there is an
in\-fi\-nite-di\-men\-sion\-al freedom of deforming it, so that
Theorem~\ref{ccdro} still applies, giving rise to compact quotient ECS
manifolds which are still incomplete, but this time not locally homogeneous. 
They belong to a wider class of compact rank-two ECS manifolds, called {\it 
di\-la\-tion\-al}. Since they are arguably of less interest than the 
lo\-cal\-ly-ho\-mo\-ge\-ne\-ous ones, we relegate their
presentation to Appendix~B.

In \cite[Theorem E]{derdzinski-terek-ms} we show that neither local
homogeneity nor the di\-la\-tion\-al property can occur for a compact rank-one
ECS manifold which satisfies a natural {\it genericity condition\/} imposed on
the Weyl tensor. In the case of our simply connected ``model'' manifolds
(Section~\ref{sd}) of dimensions $\,n\ge4$, genericity means that 
$\,\mathrm{rank}\hskip1.7ptA=n-3$, for a certain nonzero nil\-po\-tent
en\-do\-mor\-phism $\,A\,$ of an $\,(n-2)$-di\-men\-sion\-al vector space used
in constructing the model \cite[formula (6.4)]{derdzinski-terek-ms},
\cite[Remark 5.4]{derdzinski-terek-ro}. The models leading to our rank-two
examples, in odd dimensions $\,n\ge5$, all have
$\,\mathrm{rank}\hskip1.7ptA=1$. See
formula (\ref{aem}). They thus represent the maximum extent of
non\-ge\-ner\-ic\-i\-ty possible in the category of nonzero nil\-po\-tent
en\-do\-mor\-phisms.

We do not know whether lo\-cal\-ly-ho\-mo\-ge\-ne\-ous (or di\-la\-tion\-al)
compact ECS manifolds exist in any {\it even\/} dimension $\,n\ge4$. However,
if they 
do, they cannot be constructed by the same method as our odd-di\-men\-sion\-al
examples. Namely, as we observe at the end of Section~\ref{zs}, for every
$\,\bbZ$-spec\-tral system $\,(m,k,E,\xs)$, the integer $\,m$, corresponding
to the dimension $\,n=m+2$, is necessarily odd.

\renewcommand{\thetheorem}{\thesection.\arabic{theorem}}
\section{Preliminaries}\label{pr}
\setcounter{equation}{0}
\begin{lemma}\label{qpqiz}Let\/ $\,q\in(0,\infty)\smallsetminus\{1\}\,$ and\/
$\,q+q\nh^{-\nnh1}\nh\in\bbZ$. If\/ $\,\gy_0\w,\dots,\gy_m\w$ are powers of
$\,q\,$ with integer exponents, forming pairs of mutual inverses, including
the value $\,1\,$ as its own inverse when $\,m\,$ is even, then\/
$\,\gy_0\w,\dots,\gy_m\w$ form the spectrum of a matrix in\/ 
$\,\mathrm{GL}\hh(m+1,\bbZ)$.
\end{lemma}
\begin{proof}It suffices, cf.\ \cite[p.\ 75]{derdzinski-roter-10}, to show
that $\,\gy_0\w,\dots,\gy_m\w$ are the roots of a degree $\,m+1\,$ polynomial
with integer coefficients which has the leading coefficient $\,(-\nnh1)^{m+1}$
and the constant term $\,1$. This is immediate if $\,m=0\,$ and 
$\,\gy_0\w=1$, or $\,m=1\,$ and $\,(\gy_0\w,\gy_1\w)=(q,q\nh^{-\nnh1})$,
or $\,m=1\,$ and $\,(\gy_0\w,\gy_1\w)=(q^a\nh,q\nh^{-\nh a})\,$ with any
positive integer $\,a$. (The last claim is a well-known consequence of 
the preceding one, since $\,q^a\nh+q\nh^{-\nh a}$ equals a specific monic
degree $\,a\,$ polynomial with integer coefficients, evaluated on 
$\,q+q\nh^{-\nnh1}\nh$.) The required degree $\,m+1\,$ polynomial is the
product of the quadratic (and possibly linear) ones arising as above
when $\,m=0\,$ or $\,m=1$.
\end{proof}
\begin{remark}\label{semnu}We call a pseu\-\hbox{do\hs-}Euclid\-e\-an inner
product $\,\lr\,$ on an $\,m$-di\-men\-sion\-al real vector space $\,\mv\nnh$ 
{\it sem\-i-neu\-tral\/} if its positive and negative indices differ by at 
most one. Clearly, the matrix representing $\,\lr\,$ in a suitable basis 
$\,e\nh_1\w,\dots,e\nh_m\w$ of $\,\mv\hs$  
has zero entries except those on the main anti\-di\-ag\-o\-nal, all equal
to some sign factor $\,\ve=\pm1$, which for even $\,m\,$ may be assumed equal
to $\,1$, but is unique for odd $\,m$, as it then equals the difference of 
the two %positive and negative
indices. Equivalently,
\begin{equation}\label{eie}
\langle e\nh_i\w,e\nh_k\w\rangle=\ve\delta\nh_{i\hn j}\w\mathrm{\ for\ all\ 
}\,i,j\in\{1,\dots,m\}\mathrm{,\ where\ }\,k=m+1-j\mathrm{,\ and\ 
}\,\ve=\pm1\hh.
\end{equation}
Given $\,\mv\nnh,\lr,e\nh_1\w,\dots,e\nh_m\w$ as above,
$\,q\in\hn(0,\infty)$, and $\,(a(1),\dots\hn,a(m))\in\mr$ with
\begin{equation}\label{aoe}
a(1)=1\,\mathrm{\ \ and\ \ }\,a(i)+a(j)=0\,\mathrm{\ whenever\ }\,i+j=m+1\hh,
\end{equation}
we define a nonzero, trace\-less, $\,\lr$-self-ad\-joint linear
en\-do\-mor\-phism $\,A\,$ of $\,\mv\hs$ and a linear $\,\lr$-isom\-e\-try
$\,\bc:\mv\nh\to\mv\hs$ such that
$\,\bc\hskip-1.8ptA\bc^{-\nnh1}\nnh=\hs q^2\hskip-2.3ptA\,$ by setting
\begin{equation}\label{aem}
Ae\nh_m\w=e\nh_1\w\hh,\hskip8ptAe\nh_i\w=0\,\mathrm{\ if\ }\,i<m\hh,\hskip8pt
\mathrm{\ and\ }\,\bc e\nh_i\w=q^{a(i)}\nh e\nh_i\w\mathrm{\ for\ all\ }\,i\hh.
\end{equation}
In fact, as $\,\langle Ae\nh_m\w,e\nh_m\w\rangle\,$ is the only nonzero entry of
the form $\,\langle Ae\nh_i\w,e\nh_j\w\rangle$, the matrix of $\,A$ in our
basis has zeros on the diagonal; 
$\,\bc\hskip-1.8ptA\bc^{-\nnh1}\nh e\nh_i\w$ and $\,q^2\hskip-2.3ptAe\nh_i\w$
are both zero if $\,i<m$ and both $\,q^2\nh e\nh_1\w$ when $\,i=m$,
while the spans of $\,e\nh_i\w,e\nh_j\w$ with $\,i+j=m+1$ form an orthogonal
decomposition of $\,\mv\hs$ into Lo\-rentz\-i\-an planes (and a line, for
odd $\,m\,$ and $\,i=j=(m+1)/2$), and in each plane $\,\bc\,$ acts as a 
Lo\-rentz\-i\-an boost.
\end{remark}
We phrase two more obvious facts as remarks, for easy reference.
\begin{remark}\label{eigli}Every family of eigen\-vec\-tors of an
en\-do\-mor\-phism of a vector space, corresponding to mutually distinct
eigen\-val\-ues, is linearly independent.
\end{remark}
\begin{remark}\label{algeq}If the $\,s\,$ characteristic roots of an 
en\-do\-mor\-phism $\,\varPi\,$ of an $\,s$-di\-men\-sion\-al real vector 
space $\,\mathcal{Y}\,$ are all real, distinct, and form the spectrum of a
matrix $\,\varXi\,$ in $\,\mathrm{GL}\hh(s,\bbZ)$, then
$\,\varPi(\Sigma)=\Sigma\,$ for some lattice $\,\Sigma\,$ in
$\,\mathcal{Y}\nh$. (This is true for $\,\varPi=\varXi\,$ and
$\,\mathcal{Y}=\bbR\nh^s\nh$, with $\,\Sigma=\bbZ\hn^s\nh$. The general case 
follows as the algebraic equivalence
type of a di\-ag\-o\-nal\-iz\-able 
en\-do\-mor\-phism is %uniquely
determined by its spectrum.)
\end{remark}

\section{$\bbZ$-spec\-tral systems}\label{zs}
\setcounter{equation}{0}
By a $\,\bbZ\nh${\it-spec\-tral system\/} we mean a quadruple $\,(m,k,E,\xs)\,$
consisting of integers $\,m,k\ge2$, an injective function
$\,E:\vr\to\bbZ\smallsetminus\{-\nnh1\}$, where $\,\vr\nh=\{1,\dots,2m\}$, and
a function $\,\xs:\vr\to\{0,1\}$, satisfying the following conditions for all 
$\,i,i'\nh\in\vr\nh$.
\begin{enumerate}
\item[(a)]$k+1=2E(1)\,$ (and so $\,k\,$ must be odd).
\item[(b)]$E(i)+E(i')=-\nnh1\,$ and $\,\xs(i)\ne\xs(i')\,$ whenever
 $\,i+i'\nh=2m+1$.
\item[(c)]$E(i)-E(i')=k\,$ and $\,\xs(i)\ne\xs(i')\,$ if $\,i'\nh=i+1\,$ is
even.
\item[(d)]The set $\,Y\nh=\{-\nnh1\}\cup\{E(i);i\in\vr\hs\mathrm{\ and\ 
}\xs(i)=1\}\,$ is symmetric about $\,0$.
\end{enumerate}
In terms of the pre\-im\-age $\,\xl=\xs^{-\nnh1}(1)=\{i\in\vr:\xs(i)=1\}$, the 
requirements imposed on $\,\xs\,$ state that $\,\xl\,$ is a simultaneous
selector for the two families,
\begin{equation}\label{sel}
\{\{i,i'\}\in\mathcal{P}\hskip-2pt_2\w(\vr):i+i'\nh=2m+1\}\hh,\hskip7pt
\{\{i,i'\}\in\mathcal{P}\hskip-2pt_2\w(\vr):i'\nh=i+1\,\mathrm{\ is\
even}\}\hh,
\end{equation}
of pairwise disjoint $\,2$-el\-e\-ment sub\-sets of $\,\vr\nh$, while 
$\,\xs\,$ equals the characteristic function of $\,\xl$. Here 
$\,\mathcal{P}\hskip-2pt_2\w(\vr)\,$ denotes the family of all
$\,2$-el\-e\-ment sub\-sets of $\,\vr\nh$. Thus, as $\,E\,$ was assumed
injective, with $\,|\hskip2.3pt|\,$ standing for cardinality,
\begin{equation}\label{mel}
|\xl|=|E(\xl)|
=m\hh,\hskip12ptY\nh=\{-\nnh1\}\cup E(\xl)\hh,\hskip12pt|\hh Y\nh|=m+1\hh.
\end{equation}
\begin{remark}\label{relev}What makes $\bbZ$-spec\-tral systems relevant for 
our purposes is the fact that, given any such system $\,(m,k,E,\xs)\,$ and any 
$\,q\in(0,\infty)\smallsetminus\{1\}\,$ with $\,q+q\nh^{-\nnh1}\nnh\in\bbZ$,
the $\,(m+1)$-el\-e\-ment set $\,\{q^a:a\in Y\}\,$ forms, according to
Lemma~\ref{qpqiz}, the spectrum of a matrix in $\,\mathrm{GL}\hh(m+1,\bbZ)$.
\end{remark}
\begin{theorem}\label{zspsy}There exist\/ $\,\bbZ\nh$-spec\-tral systems\/ 
$\,(m,k,E,\xs)\,$ having\/ $\,k=m+2$, which realize all odd values of\/
$\,m\ge3$. Specifically, for\/ $\,m=2\hh r-3\,$ and\/ $\,k=2\hh r-1$, with
any given integer $\,r\ge3$, writing\/ $\,(i,i')=(2j-1,2j)\,$ whenever\/ 
$\,i,i'\nh\in\vr\hs$ and\/ $\,i'\nh=i+1\,$ is even, we may set
\[
(E(2j-1),E(2j))\,
=\,\begin{cases}{(r,\,-r+1)\,\,}&\mathrm{if\ }j=1$$\hskip.4pt,\cr
{(j-1,\,-\nh2\hh r+j)\,\,}&\mathrm{if\ }1<j<r-1\,\mathrm{\ and\ 
}\,r\,\mathrm{\ is\ even}$$\hskip.4pt,\cr
{(2\hh r+j-2,\,j-1)\,\,}&\mathrm{if\ }1<j<r-1\,\mathrm{\ and\ 
}\,r\,\mathrm{\ is\ odd}$$\hskip.4pt,\cr
{(r-1,\,-r)\,\,}&\mathrm{if\ }j=r-1$$\hskip.4pt,\cr
{(j-2\hh r+2,\,j-4r+3)\,\,}&\mathrm{if\ }r-1<j<m\,\mathrm{\ and\ 
}\,r\,\mathrm{\ is\ odd}$$\hskip.4pt,\cr
{(j+1,\,j-2\hh r+2)\,\,}&\mathrm{if\ }r-1<j<m\,\mathrm{\ and\ 
}\,r\,\mathrm{\ is\ even}$$\hskip.4pt,\cr
{(r-2,\,-r-1)\,\,}&\mathrm{if\ }j=m$$\hskip.4pt,\cr
\end{cases}
\]
and declare\/ $\,\xs(i)\,$ to be\/ $\,1\,$ or\/ $\,0\,$
depending on whether\/ $\,E(i)\,$ is odd or even, so that\/ $\,Y\nnh$ in\/
{\rm(d)} obviously
consists of\/ $\,-\nnh1\,$ and all values of\/ $\,E\,$ which are odd\/{\rm:}
\begin{equation}\label{wye}
Y\nh=\{-\nnh1\}\cup[\bbZ_{\mathrm{odd}}\w\cap E(\vr)]\hh,\mathrm{\ \ where\
}\,\bbZ_{\mathrm{odd}}\w=\bbZ\smallsetminus2\bbZ\hh.
\end{equation}
Also, $\,Y\nnh$ is the intersection of\/ $\,\bbZ_{\mathrm{odd}}\w$ with one,
or a union of three, closed intervals\/{\rm:}
\begin{equation}\label{uni}
\begin{array}{l}
Y\nh
=\bbZ_{\mathrm{odd}}\w\cap[-\nh2\hh r+3,2\hh r-3]\,\mathrm{\ 
for\ even\ }\,r\,\mathrm{\ while,\ if\ }\,r\,\mathrm{\ is\ odd,}\\
Y\nh=\bbZ_{\mathrm{odd}}\w\cap([-3\hh r+4,-\nh2\hh r-1]\cup[-r,r]
\cup[2\hh r+1,3r-4])\hh.
\end{array}
\end{equation}
\end{theorem}
\begin{proof}Once we establish (a) -- (c) for $\,E$, (b) -- (c)
for $\,\xs\,$ will follow: as
$\,E(i),E(i')\,$ in (a) -- (c) have different parities,
$\,\xs(i)\ne\xs(i')\,$ in $\,\{0,1\}$. 

Since $\,k=2\hh r-1$, (a) and (c) for $\,E\,$ are immediate, with 
$\,(i,i')=(2j-1,2j)$. To verify (b) for $\,E\,$ we display the definition of
$\,E\,$ in the matrix form:
\[
\left[\begin{matrix}1&2\cr
2j-1&2j \cr
2j-1&2j \cr
2\hh r-3&2\hh r-2\cr
2j'\nh-1&2j'\cr
2j'\nh-1&2j'\cr
2m-1&2m\end{matrix}\right]
\mapsto\left[\begin{matrix}E(1)&E(2)\cr
E(2j-1)&E(2j)\cr
E(2j-1)&E(2j)\cr
E(2\hh r-3)&E(2\hh r-2)\cr
E(2j'\nh-1)&E(2j')\cr
E(2j'\nh-1)&E(2j')\cr
E(2m-1)&E(2m)\end{matrix}\right]\nnh
=\nnh\left[\begin{matrix}r&-r+1\cr
j-1&-\nh2\hh r+j\cr
2\hh r+j-2&j-1\cr
r-1&-r\cr
j'\nh-2\hh r+2&j'\nh-4r+3\cr
j'\nh+1&j'\nh-2\hh r+2\cr
r-2&-r-1\end{matrix}\right]\nnh\nnh,
\]
where rows 3 and 5 (or, 2 and 6) are to be ignored if $\,r\,$ is
even (or, odd), while the ranges of $\,j\,$ and $\,j'$ are $\,1<j<r-1\,$ and
$\,r-1<j'\nh<m$.

In the first matrix above two entries have the sum $\,2m+1=4r-5\,$ if and
only if they lie symmetrically about the center of the matrix 
rectangle, with $\,j+j'\nh=2(r-1)$ (that is, with $\,j\,$ and $\,j'$
lying symmetrically about $\,r-1$). The same pairs of entries in the 
third matrix above have the sum $\,-\nnh1$, proving (b) for $\,E$.
Next,
\begin{equation}\label{rng}
\mathrm{the\ range\ }\,E(\vr)\,\mathrm{\ contains\ }\,\{1,\dots,r\}\hh,
\end{equation}
as the values $\,r-2,r-1,1\,$ appear in the first column %, rows 1, 5, 7,
of the third matrix, and $\,j-1\,$ for $\,j=2,\dots,r-2\,$ in row 2 or 3,
depending on parity of $\,r$. Also,
\begin{equation}\label{atl}
E(\vr)\,\mathrm{\ includes\ the\ }\,r-3\,\mathrm{\ values\
}\,\begin{cases}{r+1,\dots,2\hh r-3}&\mathrm{if\
}\,r\,\mathrm{\ is\ even}$$\hskip.4pt,\cr
{2\hh r,2\hh r+1,\dots,3r-4}&\mathrm{if\ }\,r\,\mathrm{\ is\
odd}$$\hskip.4pt.\cr
\end{cases}
\end{equation}
Namely, for even $\,m\,$ we get $\,j'\nh+1\,$ (row 6, with
$\,j'\nh=r,\dots,m-1=2\hh r-4$) while, if $\,r\,$ is odd, row 3 provides
$\,2\hh r+j-2$, with $\,j=2,\dots,r-2$. In addition, by (b),
\begin{equation}\label{cls}
E(\vr)\,\mathrm{\ is\ closed\ under\ the\ reflection\
}\,i\mapsto-\hs i-1\,\mathrm{\ about\ }\,-\hskip-3pt1/2\hh.
\end{equation}
Due to (\ref{rng}) -- (\ref{atl}), $\,E(\vr)\,$ contains at 
least $\,m=2\hh r-3\,$ positive integers, and -- according to (\ref{cls}) --
at least as many negative ones. Since $\,\vr\,$ has the cardinality $\,2m$,
injectivity of $\,E\,$ follows, and `at least' in the last sentence amounts to
{\it exactly}. Thus, from (\ref{rng}) -- (\ref{cls}), $\,E(\vr)\,$ is
the intersection of $\,\bbZ\,$ with the union of two or four closed intervals:
$\,[-\nh2\hh r+2,-\nh2]\cup[1,2\hh r-3]\,$ for even $\,r$, or, if $\,r\,$ is
odd,
\[
[-3\hh r+3,-\nh2\hh r-1]\cup[-r-1,-\nh2]\cup[1,r]\cup[2\hh r,3r-4]\hh.
\]
In particular, $\,-\nnh1\notin E(\vr)$. Finally, (\ref{uni}) is a trivial
consequence of this last description of $\,E(\vr)\,$ and (\ref{wye}), and it
clearly implies symmetry of $\,Y\nnh$ about $\,0$.
\end{proof}
In every $\,\bbZ$-spec\-tral system $\,(m,k,E,\xs)$, {\it the integer\/
$\,m\,$ must be odd}. Namely, since $\,\xl\,$ has the
si\-mul\-ta\-ne\-ous-se\-lec\-tor property -- see the line preceding
(\ref{sel}) -- when $\,i\in\xl\,$ is odd (or, even), 
$\,2m+1-i\in\vr\smallsetminus\xl\,$ will be even (or, odd), and so 
$\,2m-i\in\xl$ (or, respectively, $\,2m+2-i\in\xl$). In other words,
if $\,i\in\xl$, then $\,\{i,i'\}\subseteq\xl\,$ for the unique
$\,i'\nh\in\vr\hs$ having $\,i\equiv i'\hs\,$mod\hskip2pt$2\,$ and
$\,i+i'\nh=2m+1+(-\nnh1)^i\nh$. The resulting
sets $\,\{i,i'\}\,$ form a partition of $\,\xl\,$ and, clearly,
$\,i'\nh\ne i\,$ unless $\,m\,$ is odd (with $\,i'\nh=i\,$ equal to $\,m\,$ or
$\,m+1$). If $\,m\,$ were even, the $\,m$-el\-e\-ment set $\,\xl$ would thus
be partitioned into our $\,m/2\,$ disjoint $\,2$-el\-e\-ment sets
$\,\{i,i'\}$. %With $\,\equiv\,$ denoting congruence mod\hskip2pt$2$,
Oddness of $\,k$, due to (a), and (b) -- (c) would now give 
$\,E(i)\equiv E(i')\,\,$mod\hskip2pt$2\,$ on each such $\,\{i,i'\}$, making 
$\,\sum_{i\in\xl}\w E(i)\,$ even. This contradicts the equality 
$\,\sum_{i\in\xl}\w E(i)=1$, immediate, since
$\,E:\vr\to\bbZ\smallsetminus\{-\nnh1\}$ is injective, from symmetry about
$\,0\,$ of $\,Y\nh=\{-\nnh1\}\cup E(\xl)$, cf.\ (d) and (\ref{mel}).

\section{Standard di\-la\-tion\-al models}\label{sd}
\setcounter{equation}{0}
We define a {\it standard di\-la\-tion\-al ECS model\/} to be an 
$\,n$-di\-men\-sion\-al \hbox{pseu\-do\hs-} Riem\-ann\-i\-an manifold
\begin{equation}\label{met}
(\hm\nh,\hg)\,
=\,((0,\infty)\times\bbR\times\mv\nh,\,\kappa\,dt^2\nh+\,dt\,ds\hs+\hs\lr)\hh,
\end{equation}
built from the data $\,q,n,\mv\nh,\lr,A,\bc,f\,$ consisting of a real number 
$\,q\in(0,\infty)\smallsetminus\{1\}$ with $\,q+q\nh^{-\nnh1}\nnh\in\bbZ$,
an integer $\,n\ge4$, a real vector space $\,\mv\hh$ of dimension $\,n-2$,
a pseu\-do\hs-Euclid\-e\-an inner product $\,\lr\,$ on $\,\mv\nh$, a nonzero,
trace\-less, $\,\lr$-self-ad\-joint linear operator $\,A:\mv\nh\to\mv\nh$,
a linear $\,\lr$-isom\-e\-try $\,\bc:\mv\nh\to\mv\nh$, and a nonconstant
$\,C^\infty\nh$ function $\,f:(0,\infty)\to\hs\bbR$, satisfying the conditions
\begin{equation}\label{cnd}
\mathrm{a)}\hskip7pt\bc\hskip-1.8ptA\bc^{-\nnh1}\nnh
=\hs q^2\hskip-2.3ptA\hh,\hskip12pt\mathrm{b)}\hskip7ptf(t)
=q^2\nnh f(qt)\,\,\mathrm{\ for\ all\ }\,t\in(0,\infty)\hh.
\end{equation}
In (\ref{met}) we identify $\,dt,\,ds\,$ and the flat metric $\,\lr\,$ on
$\,\mv\hs$ with their pull\-backs
to $\,\hm\nh$, the function $\,\kappa:\hm\nh\to\bbR\,$ is defined by 
$\,\kappa(t,s,v)=f(t)\hskip.4pt\langle v,v\rangle+\langle Av,v\rangle$,
and $\,(t,s)$ are the Cartesian coordinates on $\,(0,\infty)\times\bbR$.

It is well known \cite[Theorem~4.1]{derdzinski-roter-09},
\cite[Sect.\,1]{derdzinski-terek-cm} that (\ref{met}) is an ECS manifold,
having rank one if $\,\mathrm{rank}\hs A>1$, and rank two when
$\,\mathrm{rank}\hs A=1$. 

The following text leading up to formula (\ref{cte}) repeats, almost verbatim,
some material from
\cite[Sect.\,6]{derdzinski-terek-ro}, albeit in a special case characterized
by (\ref{dia}); \cite[Sect.\,6]{derdzinski-terek-ro} also serves as a
reference for it, and (\ref{met}) stands, in the rest of this section, for the 
standard di\-la\-tion\-al model associated with fixed data 
$\,q,n,\mv\nh,\lr,A,\bc,f\nh$.

We denote by 
$\,\wv\hs$ and $\,\xe\hs$ the vector spaces of dimensions 
$\,2\,$ and $\,2(n-2)$ consisting of all $\,C^2\nh$ 
functions $\,\y:(0,\infty)\to\bbR$, or $\,u:(0,\infty)\to\mv\nh$, such that
\begin{equation}\label{ode}
\mathrm{i)}\hskip6pt\ddot\y\,=\,f\y\,\mathrm{\ \ or,\ respectively,\ \
}\mathrm{ii)}\hskip6pt\ddot u=f\nh u+Au\hh,\hskip9pt\mathrm{where\
}\,(\hskip2.3pt)\hskip-2.2pt\dot{\phantom o}\nh=\hs d/dt\hh.
\end{equation}
Let the operator $\,T\,$ act on functions $\,(0,\infty)\ni t\mapsto u(t)$,  
valued anywhere, by
\begin{equation}\label{tro}
[T\nnh u](t)\,=\,u(t/q)\hh,\hskip8pt\mathrm{so\ that\
(\ref{cnd}}\hyp\mathrm{b)\ reads\
}\,T\nnh f=q^2\nnh f.
\end{equation}
Thus, $\,T\,$ obviously preserves $\,\wv\nh$. We now {\it impose on\/
$\,f\,$ an additional requirement\/}:
\begin{equation}\label{dia}
\begin{array}{l}
T:\wv\hs\to\wv\,\mathrm{\ has\ %is\ di\-ag\-o\-nal\-iz\-able\ with\ 
two\ distinct\ eigen\-values\ }\,\mu^\pm\nnh\in(0,\infty)\,\mathrm{\ with\
positive}\\
\mathrm{eigen\-functions\ }\,\y\nh^+\nh\nnh,\y^-\nnh\in\nh\wv\nh\mathrm{,\ so\
that\ }\,T\nnh\y^\pm\nh=\mu^\pm\nnh\y^\pm\mathrm{\ \ and\ \
}\,\mu\nh^+\nnh\mu^-\nh=\hs\,q\nh^{-\nnh1}\nh,
\end{array}
\end{equation}
the last equality ($\det\hs T=q\nh^{-\nnh1}$ in $\,\wv\hh$) 
being immediate since the formula $\,\alpha(\y\nh^+\nnh,\y^-)
=\dot\y\nh^+\y^-\nnh-\y\nh^+\dot \y^-\hs$ (a constant!) defines an area form 
$\,\alpha\,$ on $\,\wv$ and $\,q\hs T^*\nh\alpha=\alpha$. 
The space $\,\xe\hs$ is not, in general, preserved either by $\,T\,$ 
or by $\,\bc\,$ acting val\-ue\-wise via $\,u\mapsto\bc u$, but the 
composition $\,\bc\hh T=T\hn\bc\,$ clearly leaves $\,\xe$ invariant, leading
to
\begin{equation}\label{cte}
\mathrm{\ the\ operator\ }\,\bc\hh T:\xe\nh\to\xe\hs\mathrm{\ given\ by\
}\,[(\bc\hh T)u](t)\,=\,\bc u(t/q)\hh.
\end{equation}
Next, given $\,(\hat r,\hat u),(r,u)\in\bbR\times\xe\nh$, we define mappings 
$\,\hga,\gamma:\hm\to\hm\,$ by
\begin{equation}\label{act}
\begin{array}{l}
\hga(t,s,v)
=\hs(qt,\hh-\langle\hat w(qt),\hh 2\hh\bc v
+\hat u(qt)\rangle+\hat r+s/q,\hh\bc v+\hat u(qt))\hh,\\
\gamma(t,s,v)
=\hs(t,\hh-\langle\dot u(t),\hh 2v+u(t)\rangle+r+s,v+u(t))\hh.
\end{array}
\end{equation}
where $\,\hat w=d\hat u/dt$. 
Both $\,\hga,\gamma\,$ lie in the isom\-e\-try group
$\,\mathrm{Iso}\hs(\hm\nh,\hg)\,$ \cite[formula (4.7)]{derdzinski-terek-ro}. 
We choose to treat $\hs(\hat r,\hat u)\in\bbR\times\xe\hs$ as fixed, while
allowing $\,(r,u)\,$ to range over $\,\bbR\times\xe\nh$. 
The set of all $\,\gamma\,$ arising via (\ref{act}) from all
$\,(r,u)\in\bbR\times\xe\hs$ forms a normal sub\-group 
$\,\hp\,$ of $\,\mathrm{Iso}\hs(\hm\nh,\hg)\,$
\cite[formula (4.8)]{derdzinski-terek-ro} and, as explained below,
\begin{equation}\label{oph}
\begin{array}{rl}
\mathrm{i)}&(r,u)(r'\nh,u')=(\varOmega(u',u)+r+r',\,u+u')\hh,\\
\mathrm{ii)}&\varPi(r,u)\,
=\,(2\hh\varOmega(\bc\hh T\nnh u,\hat u)\,+\,r/q,\,\bc\hh T\nnh u)\mathrm{,\
where}\\
\mathrm{iii)}&\varOmega:\xe\hn\times\hh\xe\nh\to\bbR\,\mathrm{\ is\ the\
symplectic\ form\ given\ by}\\
&\varOmega(u_1\w,u_2)=\langle\dot u_1\w,u_2\w\rangle
-\langle u_1\w,\dot u_2\w\rangle\mathrm{,\ and}\\
\mathrm{iv)}&(\bc\hh T)^*\nnh\varOmega=q\nh^{-\nnh1}\nh\varOmega\,\mathrm{\
for\ the\ operator\ }\,\bc\hh T:\xe\nh\to\xe\hs\mathrm{\ in\ (\ref{cte}).}
\end{array}
\end{equation}
Here (\ref{oph}-i) describes the group operation of $\,\hp\,$ under the
obvious identification $\,\hp=\bbR\times\xe\nh$, cf.\
\cite[(a) in Sect.\,4]{derdzinski-terek-ro}, the 
linear operator $\,\varPi:\bbR\times\xe\nh\to\bbR\times\xe\hs$ in 
(\ref{oph}-ii) equals 
$\,\hp\ni\gamma\mapsto\hga\gamma\hga\hs^{-\nnh1}\nh\in\hp$, 
the conjugation by $\hga$, cf.\ \cite[Remark 4.2]{derdzinski-terek-ro}, 
(\ref{oph}-iii) is immediate as self-ad\-joint\-ness of $\,A\,$ and
(\ref{ode}-ii) imply constancy of $\,\varOmega(u_1\w,u_2)$, and (\ref{oph}-iv)
is a consequence of (\ref{cte}).

Consider the following conditions imposed on two objects, $\,\lz\,$ and
$\,\Sigma$, with $\,\varPi$ as in (\ref{oph}-ii) 
for our fixed $\,(\hat r,\hat u)\in\bbR\times\xe\nh$.
\begin{equation}\label{ace}
\begin{array}{rl}
\mathrm{(A)}&\lz\subseteq\xe\hs\mathrm{\ is\ a\ vector\ sub\-space\ of\
dimension\ }\,n-2\hh.\\
\mathrm{(B)}&\bc\hh T\,\mathrm{\ in\ (\ref{cte})\ leaves\ }\,\lz\,\mathrm{\
invariant.}\\
\mathrm{(C)}&\Sigma\,\mathrm{\ is\ a\ (full)\ lattice\ in\
}\,\bbR\times\lz\,\mathrm{\
and\ }\,\varPi(\Sigma)=\Sigma\hh.\\
\mathrm{(D)}&\varOmega(u,u')=0\,\mathrm{\ whenever\ }\,u,u'\in\lz\mathrm{,\ 
with\ }\,\hs\varOmega\,\mathrm{\ as\ in\ (\ref{oph}{\hyp}iii).}\\
\mathrm{(E)}&u\mapsto u(t)\,\mathrm{\ is\ an\ iso\-mor\-phism\
}\,\lz\to\mv\hs\mathrm{\ for\ every\ }\,t\in(0,\infty)\hh.
\end{array}
\end{equation}
Our choice of symbols has obvious reasons: $\,\hp\,$ is a Heis\-en\-berg
group, and $\,\lz\,$ a La\-grang\-i\-an sub\-space of $\,\xe\nh$.

The following remark and lemma use the hypotheses preceding (\ref{ace}).
\begin{remark}\label{vsadd}As an obvious consequence of (\ref{oph}-i),
whenever a vector sub\-space $\,\lz\subseteq\xe\hs$ satisfies
(\ref{ace}\hs-\hn D), 
$\,\bbR\times\lz\,$ is an Abel\-i\-an 
subgroup of $\,\hp=\bbR\times\xe\subseteq\mathrm{Iso}\hs(\hm\nh,\hg)$ and
the group operation in $\,\bbR\times\lz\,$ coincides with the vec\-tor-space
addition.
\end{remark}
\begin{lemma}\label{eqdif}Condition\/ {\rm(\ref{ace}\hs-\hn E)} for a vector
sub\-space\/ $\,\lz\subseteq\xe\hs$ implies that % the assignment
$\,(t,z,u)\mapsto(t,s,v)
=(t,z\,-\,\langle\dot u(t),u(t)\rangle,\,u(t))\,$ 
is %constitutes
an\/ $\,\hp$-equi\-var\-i\-ant dif\-feo\-mor\-phism\/ 
$\,(0,\infty)\times\bbR\times\lz\,\to\,\hm\nh$.
\end{lemma}
\begin{proof}This is a special case of \cite[Remark 9.1]{derdzinski-terek-ro}.
\end{proof}
\begin{remark}\label{nllpd}As pointed out in 
\cite[the lines following formula (7.2)]{derdzinski-terek-tc},
the coordinate vector field $\,\partial/\partial s\,$ in (\ref{met}) is null
and parallel. Thus, $\,\partial/\partial s\,$ spans a one-di\-men\-sion\-al 
null parallel distribution $\,\mathcal{P}\nh$, contained, according to
\cite[Sect.\,1]{derdzinski-terek-cm}, in the Ol\-szak distribution 
$\,\dz\nh$, while $\,\nabla\nh dt=0\,$ since
$\,dt=2g(\partial/\partial s,\,\cdot\,)$. The mappings 
(\ref{act}) multiply $\hs t\hs$ and its gradient
$\,2\partial\nh/\nh\partial s\,$ by constants, and so $\hs\mathcal{P}\hn$
gives rise to distributions, also 
denoted by $\hs\mathcal{P}\nh$, on the compact quotients constructed in
Sections~\ref{fc} and~\ref{lh}.
\end{remark}
\begin{remark}\label{lhinc}A standard dilational model manifold (see
Section~\ref{sd}) is never geodesically complete. Namely,
$\,\nabla\nh dt=0\,$ in Remark~\ref{nllpd}. Thus, 
$\,t\,$ restricted to any geodesic is an af\-fine function of its parameter,
and so $\,t\,$ itself serves as 
such parameter for a geodesic $\,t\mapsto x(t)\,$ through any point $\,x\,$
with an initial velocity $\,v$ at $\,x\,$ having $\,d_v\w t=1$.
Our claim follows since $\,t\,$ ranges over $\,(0,\infty)$.
\end{remark}

\section{From $\,\bbZ$-spec\-tral systems to conditions
{\rm(\ref{ace})}}\label{fz}
\setcounter{equation}{0}
Suppose that $\,(m,k,E,\xs)\,$ is a $\,\bbZ$-spec\-tral system
(Section~\ref{zs}), $\,q\in(0,\infty)\smallsetminus\{1\}$ has
$\,q+q\nh^{-\nnh1}\nnh\in\bbZ$, while a
$\,C^\infty\nh$ function $\,f:(0,\infty)\to\hs\bbR\,$ satisfies both
(\ref{cnd}-b) and (\ref{dia}) with $\,\mu^\pm\nh=q\hs^{(-\nnh1\pm k)/2}\nh$.
We set $\,n=m+2$, choose a sem\-i-neu\-tral inner
product $\,\lr\,$ on an $\,m$-di\-men\-sion\-al real vector space $\,\mv\hs$
(see Remark~\ref{semnu}), a basis $\,e\nh_1\w,\dots,e\nh_m\w$ of 
$\,\mv\hs$ satisfying (\ref{eie}), and define $\,A,\bc:\mv\nh\to\mv\hs$ by
(\ref{aem}) for $\,a(1),\dots\hn,a(m)\,$ with
\begin{equation}\label{aje}
a(j)=E(2j-1)+(1-k)/2\mathrm{,\ that\ is,\ }\,a(j)=E(2j)+(1+k)/2\hh,
\end{equation}
the equivalence of both descriptions, and (\ref{aoe}), being due to (a) -- (c) 
in Section~\ref{zs}, which also easily imply that, for our 
$\,\mu^\pm\nh=q\hs^{(-\nnh1\pm k)/2}\nh$,
\begin{equation}\label{mpm}
(\mu\nh^+\nh q^{a(1)}\nh,\,\mu^-\nh q^{a(1)}\nh,\hs\dots,
\,\mu\nh^+\nh q^{a(m)}\nh,\,\mu^-\nh q^{a(m)})\,
=\,(q^{E(1)}\nh,\dots,q^{E(2m)})\hh.
\end{equation}
According to Remark~\ref{semnu}, these data $\,q,n,\mv\nh,\lr,A,\bc,f\,$ have 
all the properties %listed in the lines
preceding (\ref{cnd}). Thus, %and hence %so that
they lead to a %the corresponding
standard di\-la\-tion\-al model $\,(\hm\nh,\hg)\,$ with (\ref{met}).
\begin{lemma}\label{cnseq}The assumptions just listed have the following
consequences.
\begin{enumerate}
\item[(a)]Some ordered basis\/
$\,(u_1^+\nh,u_1^-\nh,\dots,u_m^+\nh,u_m^-)=(u_1\w,\dots,u_{2m}\w)\,$ of
$\,\xe\nnh$ consists of eigen\-vec\-tors of\/ $\,\bc\hh T:\xe\nh\to\xe\nh$,
cf.\ {\rm(\ref{cte})}, and the respective eigen\-values, equal to\/
$\,q^{E(1)}\nh,\dots,q^{E(2m)}\nh$, are pairwise distinct. 
%\item[(b)]
With\/ $\,\y^\pm$ as in\/ {\rm(\ref{dia})} and suitable functions\/ 
$\,z^\pm\nh:(0,\infty)\to\bbR$, %and our\/
%$\,\mu^\pm\nh=q\hs^{(-\nnh1\mp k)/2}\nh$,
this basis may be obtained by
setting\/ $\,u_i^\pm=\y^\pm\nh e\nh_i\w$ if $\,i<m\,$ and\/ 
$\,u_m^\pm=\y^\pm\nh e\nh_m\w+z^\pm\nh e\nh_1\w$.
\item[(b)]$\varOmega(u_i\w,u\nh_j\w)=0\,$ %if and only
whenever\/ $\,i,j\in\{1,\dots,2m\}\,$ and\/ $\,i+j\ne2m+1$, the basis\/
$\,(u_1\w,\dots,u_{2m}\w)\,$ and\/ 
$\,\varOmega\,$ being as in\/ {\rm(a)} and\/ {\rm(\ref{oph}-iii)}.
\end{enumerate}
\end{lemma}
\begin{proof}Due to (\ref{dia}) and (\ref{aem}), 
$\,u_i^\pm$ defined in (a) have 
$\,\bc\hh Tu_i^\pm=\mu^\pm\nh q^{a(i)}\nh u_i^\pm$ if $\,i<m$, that is, by
(\ref{mpm}), $\,\bc\hh Tu\nh_j\w=q^{E(j)}\nh u\nh_j\w$ whenever
$\,j\in\{1,\dots,2m-2\}$. That $\,q^{E(1)}\nh,\dots,q^{E(2m)}$ are distinct
follows as $\,E:\vr\to\bbZ\smallsetminus\{-\nnh1\}\,$ is injective
(Section~\ref{zs}).

Given functions $\,x^\pm\nh:(0,\infty)\to\bbR\,$ with 
$\,\ddot x^\pm\nh=f\nh x^\pm\nh+\y^\pm\nh$, (\ref{ode}-i) for 
$\y=\y^\pm$ and (\ref{aem}) yield (\ref{ode}-ii) for 
$\,u^\pm\nh=\y^\pm\nh e\nh_m\w+x^\pm\nh e\nh_1\w$, so that
$\,u^\pm\in\xe\hs$ and, again by (\ref{dia}), (\ref{aem}) and (\ref{mpm}),
$\,w^+\nh=[\hh\bc\hh T\nh-\hs q^{E(2m-1)}]u^+$ and 
$\,w^-\nh=[\hh\bc\hh T\nh-\hs q^{E(2m)}]u^-$ both lie in the subspace 
$\,\mathcal{Z}\,$ of $\,\xe\hs$ spanned by the eigen\-vec\-tors
$\,u_1\w,u_2\w$ for the eigen\-val\-ues $\,q^{E(1)}\nh,q^{E(2)}\nh$. (The
scalars $\,q^{E(2m-1)}\nh,q^{E(2m)}$ stand for the corresponding multiples of
identity.) Distinctness of the eigen\-val\-ues $\,q^{E(1)}\nh,\dots,q^{E(2m)}$
implies that $\,\bc\hh T\nh-\hs q^{E(2m-1)}\,$ and 
$\,\bc\hh T\nh-\hs q^{E(2m)}\,$ map $\,\mathcal{Z}\,$ isomorphically onto
itself. We may now choose $\,z^\pm$ to be the function such that 
$\,\bc\hh T\nh-\hs q^{E(2m-(1\pm1)/2)}\,$
sends $\,(x^\pm\nh-z^\pm)\hs e\nh_1\w$ onto $\,w^\pm\nh$, and
(a) follows, with linear independence of $\,u_1\w,\dots,u_{2m}\w$ due to
Remark~\ref{eigli}.

Next, by (\ref{oph}-iv) and (a), 
$\,q\nh^{-\nnh1}\nh\varOmega(u_i\w,u\nh_j)
=\varOmega(\bc\hh Tu_i\w,\bc\hh Tu\nh_j\w)
=q^{E(i)+E(j)}\nh\varOmega(u_i\w,u\nh_j\w)$
which, in view of injectivity of $\,E\hh$ and (b) in Section~\ref{zs}, yields
%the %`if' part of
(b).% The matrix of $\,\varOmega\,$ in our basis is thus 
%anti\-di\-ag\-o\-nal, while nondegenracy of $\,\varOmega$, cf.\
%(\ref{oph}-iii), implies the entries on the main anti\-di\-ag\-o\-nal are all
%nonzero, and the `only if' part follows.
\end{proof}
We also fix a pair $\,(\hat r,\hat u)\in\bbR\times\xe\nh$, as in the 
lines following (\ref{act}), and denote by $\,\varPi\,$ the resulting
linear operator $\,\bbR\times\xe\nh\to\bbR\times\xe\hs$ of
conjugation by $\hga$, in (\ref{oph}-ii).
\begin{lemma}\label{abcde}With the data\/ 
$\,q,n,\mv\nh,\lr,A,\bc,f\,$ and\/ $\,(\hat r,\hat u),\varPi\,$ chosen as 
above, conditions\/ {\rm(\ref{ace})} hold for suitable\/ 
$\,\lz\,$ and\/ $\,\Sigma$.
\end{lemma}
\begin{proof}Let $\,\lz\,$ be the span of $\,\{u_i\w:i\in\xl\}\,$ for the
basis $\,(u_1^+\nh,u_1^-\nh,\dots,u_m^+\nh,u_m^-)=(u_1\w,\dots,u_{2m}\w)\,$ of 
$\,\xe\hs$ appearing in Lemma~\ref{cnseq}(a) and the set 
$\,\xl\,$ associated with our $\,\bbZ$-spec\-tral system $\,(m,k,E,\xs)\,$
(see Section~\ref{zs}), that is, $\,\xl=\{i\in\vr:\xs(i)=1\}$. 
Now (\ref{ace}\hs-\nh A) -- (\ref{ace}\hs-\hn B) follow since $\,m=n-2$. As 
$\,\xl\,$ is a selector for the second family of (\ref{sel}), the basis
$\,\{u_i\w:i\in\xl\}\,$ of $\,\lz\,$ has the form
\begin{equation}\label{uoe}
(u_1^{\ve(1)},\dots,u_m^{\ve(m)})\,\mathrm{\ with\ some\ signs\
}\,\ve(1),\dots,\ve(m)\hh.
\end{equation}
For each fixed $\,t\in(0,\infty)$, the operator
$\,\xe\ni u\mapsto u(t)\in\mv\hs$
sends $\,u_i^\pm$ to $\,\y^\pm\nh(t)\hs e\nh_i\w$ if $\,i<m\,$ and 
$\,u_m^\pm$ to $\,\y^\pm\nh(t)\hs e\nh_m\w+z^\pm\nh(t)\hs e\nh_1\w$, so that,
restricted to $\,\lz$, it is represented in the bases (\ref{uoe}) and
$\,e\nh_1\w,\dots,e\nh_m\w$ by an upper triangular matrix with all diagonal
entries positive in view of (\ref{dia}), which proves (\ref{ace}\hs-\hn E).
Simultaneously, $\,\xl\,$ is a selector for the first family in (\ref{sel}),
so that $\,i+j\ne2m+1\,$ if $\,u_i\w,u\nh_j\w\in\lz$. Combined with
Lemma~\ref{cnseq}(b), this yields (\ref{ace}\hs-\hn D). Finally, the existence
of a lattice $\,\Sigma\,$ required in (\ref{ace}\hs-C) is immediate from
Remarks~\ref{relev} and~\ref{algeq}.
\end{proof}
An example of a $\,C^\infty\nh$ function $\,f:(0,\infty)\to\hs\bbR\,$
having both (\ref{cnd}-b) and (\ref{dia}) for 
$\,\mu^\pm\nh=q\hs^{(-\nnh1\pm k)/2}\nh$, as required at the beginning of this
section, is provided by
\begin{equation}\label{exa}
f(t)\,=\,\frac{k^2\nnh-1}{4t^2}\mathrm{,\ \ with\
}\,\y^\pm\nh(t)\,=\,t\hs^{(1\mp k)/2}\mathrm{\ in\ (\ref{dia}).}
\end{equation}
For the resulting standard di\-la\-tion\-al model $\,(\hm\nh,\hg)$, 
%with (\ref{met}),
cf.\ the lines following (\ref{mpm}),
\begin{equation}\label{lho}
(\hm\nh,\hg)\,\mathrm{\ is\ locally\ homogeneous.}
\end{equation}
Namely, by (\ref{exa}), 
the expression (\ref{met}) for $\,\g\,$ amounts to that for the metric 
$\,g\hn^P$ in \cite[top of p.\ 170]{derdzinski-78}, our coordinate
$\,t\,$ being denoted there by $\,u\hn^1\nh$. Our Remark~\ref{semnu} now
clearly implies formula (10) in \cite[p.\ 172]{derdzinski-78} which, as stated
there, guarantees homogeneity of the metric $\,g^P$ on 
$\,(0,\infty)\times\bbR\times\mv\nh$, with 
$\,\mv\nh=\bbR^{\hskip-.6ptn-2}\nnh$.

\section{From conditions {\rm(\ref{ace})} to compact quotients}\label{fc}
\setcounter{equation}{0}
We now show that conditions (\ref{ace}) are sufficient for a standard
di\-la\-tion\-al model to admit compact isometric quotients. Specifically, let
$\,(m,k,E,\xs),q,f$, along with $\,n=m+2\,$ and
$\,\mv\nh,\lr,e\nh_1\w,\dots,e\nh_m\w,A,\bc$,  
have the properties listed at the beginning
of Section\/~{\rm\ref{fz}}, so that the data $\,q,n,\mv\nh,\lr,A,\bc,f\,$ give
rise to a standard di\-la\-tion\-al model $\,(\hm\nh,\hg)\,$ with (\ref{met}).
We denote by $\,\mathcal{P}\hs$ the one-di\-men\-sion\-al null parallel
distribution on $\,(\hm\nh,\hg)$, defined in Remark~\ref{nllpd}.
\begin{theorem}\label{ccdro}Under these assumptions, we also fix a 
pair\/ $\,(\hat r,\hat u)\in\bbR\times\xe\nh$, cf.\ the lines following\/ 
{\rm(\ref{act})}, and define\/ $\,\varPi\,$ by\/ {\rm(\ref{oph}-ii)}. If\/
$\,\lz\,$ and\/ $\,\Sigma\,$ are any objects satisfying\/ {\rm(\ref{ace})},
then
\begin{equation}\label{sgg}
\mathrm{the\ group\ }\,\Gm\nh\subseteq\mathrm{Iso}\hs(\hm\nh,\hg)\mathrm{,\
generated\ by\ }\,\hga\,\mathrm{\ appearing\ in\ (\ref{act})\ and\
}\,\Sigma\hh,
\end{equation}
acts on\/ 
$\,\hm$ freely and properly dis\-con\-tin\-u\-ous\-ly with a compact quotient
manifold\/ $\,M\nh=\hm\nnh/\hh\Gm\nh$. 
In addition, $\,M\,$ is the total space of a torus bundle 
over the circle, with the leaves of\/ $\,\mathcal{P}^\perp\hskip-2pt$ serving
as the fibres, and its fundamental group\/ $\,\Gm\,$ has no Abel\-i\-an 
subgroup of finite index, so that\/ $\,M\,$ cannot be dif\-feo\-mor\-phic to a
torus, or even covered by a torus.
\end{theorem}
\begin{proof}By Lemma~\ref{eqdif}, Remark~\ref{vsadd} and (\ref{ace}\hs-C),
the action of $\,\Sigma\subseteq\hp\,$ on each $\,t$-level
$\,\{t\}\times\bbR\times\mv\hs$ is, equi\-var\-i\-ant\-ly,
\begin{enumerate}
  \def\theenumi{{\rm\alph{enumi}}}
\item [(i)] identified with the additive action of the lattice $\,\Sigma\,$ on
$\,\bbR\times\lz$.
\end{enumerate}
Since $\,\varPi\hs$ acts on $\,\Sigma\,$ via conjugation by
$\,\hga$, cf.\ the lines following (\ref{oph}), 
\begin{enumerate}
  \def\theenumi{{\rm\alph{enumi}}}
\item [(ii)] $\Sigma\,$ is an Abel\-i\-an normal sub\-group of $\,\Gm$,
\end{enumerate}
where we again used Lemma~\ref{eqdif}, Remark~\ref{vsadd} and
(\ref{ace}\hs-C). Thus, 
any element of $\,\Gm\nh$, being a finite product of factors from the
set $\,\Sigma\cup\{\hga,\hga\hs^{-\nnh1}\}$, equals 
$\,\hga\hs^r\nh\gamma\,$ (written multiplicatively) with some $\,r\in\bbZ\,$
and $\,\gamma\in\Sigma$. From (\ref{act}), if 
$\,(t,s,v)\in\hm\nh$,
\begin{enumerate}
  \def\theenumi{{\rm\alph{enumi}}}
\item[(iii)] $(\hga\hs^r\nh\gamma)(t,s,v)=(q\hh^rt,s'\nh,v')\,$ for some  
$\,s'\nh,v'\nh$, which also leads to
\item[(iv)] the homo\-mor\-phism 
$\,\Gm\ni\hga\hs^r\nh\gamma\mapsto r\in\bbZ$,
\end{enumerate}
and so $\,\Gm\hs$ acts on $\,\hm\,$ freely: if
$\,\hga\hs^r\nh\gamma\,$ has a fixed point $\,(t,s,v)$, (iii) gives 
$\,q\hh^r\nh t=t$. Therefore, $\,r=0$, and $\,\gamma$, having a fixed
point, must equal the identity since, by (i), the action of 
$\,\Sigma\,$ on $\,\hm\,$ is free. 

Consider now sequences with the terms 
$\,(r,\gamma)\in\bbZ\times\Sigma\,$ and $\,x=(t,s,v)\in\hm$ such that
$\,x\,$ and $\,\hga^r\nh(\gamma(x))\,$ both converge. Thus, (iii) implies 
convergence of the sequence $\,r\,$ (and hence its ultimate constancy). For 
the sequences $\,\gamma'=\hga\hs^r\nh\gamma\hh\hga\hh^{-\nh r}$ in 
$\,\Sigma\,$ and $\,x'\nh=\hga^r\nh(x)\in\hm\nh$, with this 
``ultimate constant'' $\,r$, writing $\,\gamma'\nh=(r,u)\,$ 
and $\,x'\nh=(t'\nh,s'\nh,v')$,
we obtain convergence of both $\,\gamma'\nh(x')=\hga\hs^r\nh(\gamma(x))\,$
and $\,x'\nh$, so that (i) implies eventual constancy of $\,\gamma'$ and --
consequently -- that of $\,\hga\hs^r\nh\gamma\in\Gm\nh$.

The implication established in the last paragraph proves
proper dis\-con\-ti\-nu\-i\-ty of
the action of $\,\Gm\hs$ on $\,\hm\nh$. See
\cite[Exercise 12\hs-\nh19 on p.\ 337]{lee}.

Next, $\,\hm\,$ has a compact subset $\,K\,$ 
intersecting every orbit of $\,\Gm\nh$, which yields compactness of the
quotient manifold $\,M=\hm\nnh/\hh\Gm\nh$. In fact,
we may choose $\,K\,$ to be the image, under the
$\,\hp$-equi\-var\-i\-ant dif\-feo\-mor\-phism in (a), of
$\,J\times K'\nh$, where $\,J\subseteq(0,\infty)\,$ is 
the closed interval with the endpoints $\,1,q$, and $\,K'$ a 
compact set in $\,\bbR\times\lz\,$ which intersects all orbits of the lattice  
$\,\Sigma\,$ acting on $\,\bbR\times\lz\,$ by vec\-tor-space translations. We
now modify any $\,(t,s,v)\in\hm\,$ by applying to it elements of 
$\,\Gm$ twice in a row so as to end up with a point of $\,K\nh$. First, 
$\,\hga\hs^r(t,s,v)=(q\hh^rt,s'\nh,v')$, cf.\ (iii), 
has $\,q\hh^rt\in J\,$ for a suitable $\,r\in\bbZ\,$ (as the sum of 
$\,\log t\,$ and some multiple of $\,\log q\,$ lies between $\,\log q\,$ and 
$\,0$). We may thus assume that $\,t\in J$. With this fixed 
$\,t$, (i) allows us to choose $\,\gamma'\in\Sigma\,$ sending $\,(t,s,v)\,$ 
into $\,K$.

The surjective sub\-mer\-sion 
$\,\hm\ni(t,s,v)\mapsto(\log t)/(\log q)\in\bbR$, being clearly
equi\-var\-i\-ant relative to the homo\-mor\-phism (iv)  
along with the obvious actions of $\,\Gm$ on $\,\hm\nh$, via (iii), and
$\,\bbZ\,$ on
$\,\bbR$, descends to a surjective sub\-mer\-sion $\,M\to S^1$ which,
due to the compact case of Ehres\-mann's 
fibration theorem \cite[Corollary 8.5.13]{dundas}, is a bundle projection.
This leads, via (i), to the required conclusion about 
a torus bundle over the circle. The claim about the fibres follows: the leaves
of $\,\mathcal{P}^\perp\nnh$ in $\,\hm$ are the levels of $t$, since,
according to Remark~\ref{nllpd}, $\,\mathcal{P}\hs$ is 
spanned by the parallel gradient of $\,t$.

Finally, a fi\-nite-in\-dex sub\-group $\,\Gm'$ of $\,\Gm\,$ would have a
nontrivial image under the homo\-mor\-phism (iv) (the kernel of which,
$\,\Sigma$, has an infinite index in $\,\Gm\nh$, and
hence cannot contain $\,\Gm'$), and $\,\Gm'\nh\cap\Sigma\,$ would clearly be
a fi\-nite-in\-dex sub\-group of the lattice $\,\Sigma\,$ spanning,
consequently, the whole space $\,\bbR\times\lz$. The conjugation by any
$\,\gamma'\nh\in\Gm'\nh\smallsetminus\Sigma\,$ would thus lead to the
operator (\ref{oph}-ii) equal to the identity on $\,\bbR\times\lz$, and yet
having the $\,q$-component different from $\,1$. This contradiction proves the
final clause of the theorem.
\end{proof}

\section{The lo\-cal\-ly-ho\-mo\-ge\-ne\-ous case}\label{lh}
\setcounter{equation}{0}
Constructing compact rank-one ECS manifolds of dimension $\,n\,$ via
Theorem~\ref{ccdro} is clearly reduced to finding two objects: a
$\,\bbZ$-spec\-tral system $\,(m,k,E,\xs)$, for $\,m=n-2$, and a
$\,C^\infty\nh$ function $\,f:(0,\infty)\to\hs\bbR\,$ with (\ref{cnd}-b) and
(\ref{dia}), for $\,q\in(0,\infty)\smallsetminus\{1\}\,$ such that 
$\,q+q\nh^{-\nnh1}\nnh\in\bbZ$, and $\,\mu^\pm\nh=q\hs^{(-\nnh1\pm k)/2}\nh$.
One now gets the former from Theorem~\ref{zspsy}, as long as $\,n\ge5\,$ is
odd, while an example of the latter is then provided by formula (\ref{exa}). 

The resulting existence theorem may be phrased as follows.
\begin{theorem}\label{lchom}Let\/ $\,n\ge5\,$ be odd. Applying
Theorem\/~{\rm\ref{ccdro}} to data that include\/ $\,(m,k,E,\xs)\,$ of 
Theorem\/~{\rm\ref{zspsy}}, where\/ $\,m=n-2$, and\/ $\,f\,$ given by\/
{\rm(\ref{exa})}, we obtain the group\/ 
$\,\Gm\hs$ in\/ {\rm(\ref{sgg})} acting on\/ 
$\,\hm$ freely and properly dis\-con\-tin\-u\-ous\-ly with a locally
homogeneous and geodesically incomplete compact quotient rank-two ECS 
manifold\/ $\,M\nh=\hm\nnh/\hh\Gm\hs$ of dimension $\,n$, forming the total
space of a nontrivial torus bundle over the circle, with the fibres provided
by the leaves of\/ $\,\mathcal{P}^\perp\hskip-2pt$, while its fundamental
group\/ $\,\Gm\,$ has no fi\-nite-in\-dex Abel\-i\-an sub\-group.
\end{theorem}
In fact, for local homogeneity and incompleteness, see (\ref{lho}), and 
Remark~\ref{lhinc}.% below.\end{proof}

\setcounter{section}{1}
\renewcommand{\thesection}{\Alph{section}}
\renewcommand{\theequation}{\Alph{section}.\arabic{equation}}
\setcounter{theorem}{0}
\renewcommand{\thetheorem}{\thesection.\arabic{theorem}}
\section*{Appendix A: Special spectra realized in function spaces}\label{ss}
\setcounter{equation}{0}
We fix $\,q\in(0,\infty)\smallsetminus\{1\}$. For a continuous function
$\,f:(0,\infty)\to\bbR\,$ satisfying condition (\ref{cnd}-b),
recall from Section~\ref{sd} that 
the two-di\-men\-sion\-al space $\,\mathcal{W}\hs$ %consisiting of all
of $\,C^2$ solutions 
$\,\y:(0,\infty)\to\bbR\,$ to the sec\-ond-or\-der ordinary differential
equation (\ref{ode}-i) is obviously
invariant under the {\it translation operator\/} $\,T\,$ given by 
\begin{equation}\label{tra}
[T\nnh \y](t)\,=\,\y(t/q)\hh,\hskip7pt\mathrm{\ with\ }\,\det\hs T
=q\nh^{-\nnh1}\mathrm{\ in\ }\,\mathcal{W},
\end{equation}
where $\,\det\hs T=q\nh^{-\nnh1}$ as in the line following (\ref{dia}).
Clearly, (\ref{cnd}-b) amounts to 
periodicity, with the period $\,\log q$, of the 
function $\,\bbR\ni\vp\mapsto e^{2\evp}\nnh f(e^\evp)$. Therefore,
\begin{equation}\label{inf}
\begin{array}{l}
\mathrm{both\ the\ vector\ space\ }\,\hs\mathcal{F}\,\mathrm{\ of\ continuous\
functions\ }\,f\hs\mathrm{\ satisfying\ (\ref{cnd}{\hyp}b)}\\ 
\mathrm{and\,\ its\,\ subspace\,\ 
}\,\hh\mathcal{F}\nnh_0\w\hs\hs=\,\hs\{f\in\mathcal{F}:f(1)
=0\}\hh\,\mathrm{\,\ are\,\ in\-fi\-nite}\hyp\mathrm{di\-men\-sion\-al.}
\end{array}
\end{equation}
%Consider the case where $\,T\,$ associated with $\,f\,$ has, for some
We will need such $\,f\,$ with $\,T\,$ having, for some
$\,c\in(0,\infty)$, the spectrum
\begin{equation}\label{spe}
q\hh^{\pm c-1/2}\nh\mathrm{,\ that\ is,\ positive\
real\ eigen\-values\ and\ }\,\mathrm{tr}\hskip2.7ptT
=2q\nh^{-\nnh1/2}\nnh\cosh(c\log q)\hh.
\end{equation}
Examples of real-an\-a\-lyt\-ic functions $\,f\in\mathcal{F}\,$ with
(\ref{spe}) are provided by
\begin{equation}\label{fct}
f\hskip-1.9pt_c\w(t)\,=\,(c\hh^2\nh-1/4)/t^2\nh,\hskip8pt\mathrm{where\ 
}\,c\in(0,\infty)\hh.
\end{equation}
In fact, an obvious basis of $\,\mathcal{W}\hs$ for
$\,f\nh=f\hskip-1.9pt_c\w\,$ consists of $\,\y=\y_{\hn c}^\pm$ given by
\begin{equation}\label{bas}
\y_{\hn c}^\pm(t)=t^{\mp c+1/2}\mathrm{,\ so\ that\
}\,T\nnh \y_{\hn c}^\pm\nh=q^{\pm\hh c-1/2}\nh \y_{\hn c}^\pm.
\end{equation}
\begin{theorem}\label{infdm}For any fixed\/
$\,q\in(0,\infty)\smallsetminus\{1\}\,$ and\/ $\,c\in(0,\infty)\,$ there
exists an in\-fi\-nite-di\-men\-sion\-al manifold of smooth functions\/
$\,f:(0,\infty)\to\bbR\,$ with\/ {\rm(\ref{cnd}-b)} such that the
corresponding translation operator\/ $\,T:\mathcal{W}\to\hn\mathcal{W}\hs$
has the eigen\-values\/ $\,q\hh^{\pm c+1/2}\nh$, and some basis of\/
$\,\mathcal{W}\hs$ di\-ag\-o\-nal\-iz\-ing $\,T\,$ consists of positive
functions. The same remains true if one replaces `smooth\hn\nnh' by
real-an\-a\-lyt\-ic. 

More precisely, for any\/ $\,f\hskip-2.3pt_*\w\in\mathcal{F}\nnh_0\w$ -- see\/
{\rm(\ref{inf})} -- sufficiently\/ $\,C^0\nh$-close to\/ $\,0$, there exists a
unique\/ $\,a\,$ close to\/ $\,c\,$ in\/ $\,\bbR\,$ such that\/ 
$\,f=f\hskip-2.3pt_*\w+f\hskip-1.9pt_a\w$ realizes the\/
$\,T\hskip-2.6pt$-spec\-trum\/ $\,\{q^{c-1/2}\nh,q\nh^{-c-1/2}\nh\}$, while the
resulting assigment\/ $\,f\hskip-2.3pt_*\w\mapsto f\,$ is smooth and injective.
%This includes the
%function\/ $\,f\hskip-1.3pt_{3/\hn2}\w$, with\/ $\,f\hskip-2.3pt_*\w=0$.
\end{theorem}
\begin{proof}Define a mapping
$\,H:\mathcal{F}\nnh_0\w\times(0,\infty)\to\bbR\,$ by
$\,H(f\hskip-2.3pt_*\w,a)=\mathrm{tr}\hskip2.7ptT\,$ for $\,T$ arising from 
$\,f=f\hskip-2.3pt_*\w+f\hskip-1.9pt_a\w$. Smoothness of $\,H\,$ follows since
\begin{equation}\label{hfa}
H(f\hskip-2.3pt_*\w,a)=\y\nh^+(1/q)+q\nh^{-\nnh1}\hn\dot \y^-(1/q)\hh.
\end{equation}
where $\,\y\nh^+\nnh,\y^-$ are solutions to (\ref{ode}) with the initial
conditions $\,(\y\nh^+(1),\dot \y\nh^+(1))=(1,0)$ and 
$\,(\y^-(1),\dot \y^-(1))=(0,1)$. To verify (\ref{hfa}) note that any
$\,\y\in\mathcal{W}\hs$ equals
$\,\y(1)\hs\y\nh^++\dot\y(1)\hs \y^-\nh$. For $\,T\nnh\y\,$ rather than 
$\,\y\,$ this reads, by (\ref{tra}), 
$\,T\nnh\y=\y(1/q)\hs\y\nh^++q\nh^{-\nnh1}\hn\dot\y(1/q)\hs \y^-$ which,
applied to $\,\y=\y\nh^+$ and $\,\y=\y^-\nh$, gives 
\[
(T\nnh \y\nh^+\nnh,T\nnh \y^-)=(\y\nh^+(1/q)\hs \y\nh^+
+q\nh^{-\nnh1}\hn\dot \y\nh^+(1/q)\hs \y^-\nh,\,\,
\y^-(1/q)\hs \y\nh^++q\nh^{-\nnh1}\hn\dot \y^-(1/q)\hs \y^-)\hh,
\]
showing that the matrix of $\,T\,$ in the basis $\,\y\nh^+\nnh,\y^-$ of
$\,\mathcal{W}\hs$ has the trace claimed in (\ref{hfa}). 
Also, as each $\,f\hskip-1.9pt_c\w$ leads to the spectrum (\ref{spe}), 
$\,H(0,a)=2\hh q\nh^{-\nnh1/2}\nnh\cosh(a\log q)$ for all $\,a>0$, including
$\,a=c$. Since 
$\,d\hs[H(0,a)]/\hn da\ne0\,$ at $\,a=c$, the implicit mapping
theorem \cite[p.\ 18]{lang} provides neighborhoods of
$\,0\hs$ in $\,\mathcal{F}\,$ and $\,c\,$ in $\,\bbR\,$ with the required
smooth mapping $\,f\hskip-2.3pt_*\w\mapsto a\,$ sending $\,0\,$ to $\,c\,$ 
and having $\,H(f\hskip-2.3pt_*\w,a)=2\hh q\nh^{-\nnh1/2}\nnh\cosh(c\log q)$.
Injectivity of 
$\,f\hskip-2.3pt_*\w\mapsto f\hskip-2.3pt_*\w+f\hskip-1.9pt_a\w$ 
follows: $\,f(1)=f\hskip-1.9pt_a\w(1)=a\hh^2\nh-1/4$ uniquely determines
$\,a>0$, and hence $\,f\hskip-1.9pt_a\w$ and $\,f\hskip-2.3pt_*\w$ as well.

Finally, positivity of the functions (\ref{bas}) on the closed interval with
the endpoints $\,1,q\,$ yields the same for functions $\,C^0\nh$-close to them 
that di\-ag\-o\-nal\-ize $\,T\,$ for $\,f\,$ close to $\,f\hskip-1.9pt_c\w$.
Being eigen\-vectors of the translation operator $\,T\nh$, they thus remain
positive throughout $\,(0,\infty)$.
\end{proof}

\setcounter{section}{2}
\renewcommand{\thesection}{\Alph{section}}
\setcounter{theorem}{0}
\renewcommand{\thetheorem}{\thesection.\arabic{theorem}}
\section*{Appendix B: Rank-two ECS manifolds of di\-la\-tion\-al
type}\label{ro}
\setcounter{equation}{0}
The distribution $\,\mathcal{P}\hs$ (see Remark~\ref{nllpd}) on every compact
rank-two ECS manifold $\,(M\nh,\g)\,$ arising in Theorem~\ref{ccdro} is a real
line bundle over $\,M\,$ with a linear connection induced by the
Le\-vi-Ci\-vi\-ta connection of $\,\g$. Due to its obvious flatness, 
the latter connection has a 
countable holonomy group contained in $\,\bbR\smallsetminus\{0\}$.

All our examples $\,(M\nh,\g)\,$ are {\it di\-la\-tion\-al\/} in the sense
that this holonomy group is infinite, which follows since the group $\,\Gm\hs$
in (\ref{sgg}) contains the element $\,\hga\,$ defined by (\ref{act}) with
$\,q\in(0,\infty)\smallsetminus\{1\}$.

Theorem~\ref{lchom} now obviously remains valid if one replaces `given by 
(\ref{exa})' with {\it arising in Theorem\/~{\rm\ref{infdm}} for\/} $\,c=k/2$,
and `locally homogeneous' with {\it di\-la\-tion\-al\/}:
\begin{theorem}\label{dilat}Let\/ $\,n\ge5\,$ be odd. Applying
Theorem\/~{\rm\ref{ccdro}} to\/ $\,(m,k,E,\xs)\,$ of 
Theorem\/~{\rm\ref{zspsy}}, where\/ $\,m=n-2$, and\/ $\,f\,$ arising in
Theorem\/~{\rm\ref{infdm}} for\/ $\,c=k/2$, we obtain the group\/ 
$\,\Gm\hs$ in\/ {\rm(\ref{sgg})} acting on\/ 
$\,\hm$ freely and properly dis\-con\-tin\-u\-ous\-ly with a di\-la\-tion\-al 
and geodesically incomplete compact quotient rank-two ECS 
manifold\/ $\,M\nh=\hm\nnh/\hh\Gm\hs$ of dimension $\,n$, forming the total
space of a nontrivial torus bundle over the circle, the fibres of which are 
the leaves of\/ $\,\mathcal{P}^\perp\hskip-2pt$, and the fundamental 
group\/ $\,\Gm$ of\/ $\,M\,$ 
has no fi\-nite-in\-dex Abel\-i\-an sub\-group.
\end{theorem}
Geodesic incompleteness is immediate here from Remark~\ref{lhinc}.
Also, most of the examples resulting from Theorem~\ref{dilat} have the
di\-la\-tion\-al property without local homogeneity, which is 
guaranteed by the in\-fi\-nite-di\-men\-sion\-al freedom of choosing $\,f\,$
in Theorem~\ref{infdm}: in the lo\-cal\-ly-ho\-mo\-ge\-ne\-ous case 
$\,|f(t)|^{-\nh1/2}$ must be -- according to
\cite[formula (3.3)]{derdzinski-terek-ne} 
-- an af\-fine function of $\,t$. This restricts it to a 
fi\-nite-di\-men\-sion\-al moduli space.

%\newpage


\begin{thebibliography}{99}

\bibitem{cahen-kerbrat}M.\hskip2.3ptCahen and Y\nnh.\hskip2.3ptKerbrat, {\em
Transfor\-ma\-tions con\-formes des espaces sy\-m\'e\-triques
pseudo-rie\-manniens}, Ann.\ Mat.\ Pura Appl. (4) \textbf{132} (1982), 
275--289.

\bibitem{derdzinski-78}A.\hskip2.3ptDerdzi\'nski, {\em
On homogeneous con\-for\-mal\-ly symmetric
pseu\-\hbox{do\hskip.7pt-}Riem\-ann\-i\-an manifolds}, Colloq.\ Math.
\textbf{40} (1978), no.\,1, 167--185.

\bibitem{derdzinski-roter-77}A.\hskip2.3ptDerdzi\'nski and W.\hskip2.3ptRoter,
{\em On con\-for\-mal\-ly symmetric manifolds with metrics of indices $\,0\,$
and $\,1\hh$}, Tensor (N.\hskip1.6ptS.) \textbf{31} (1977), no.\,3, 255--259.

\bibitem{derdzinski-roter-07}A.\hskip2.3ptDerdzinski and W.\hskip2.3ptRoter, 
{\em Global properties of indeﬁnite metrics with parallel Weyl tensor},
in: Pure and Applied Differential Geometry - PADGE 2007, eds.\ F.\ Dillen and
I.\ Van de Woestyne, Berichte aus der Mathematik, Shaker Verlag, Aachen, 2007,
63--72.

\bibitem{derdzinski-roter-09}A.\hskip2.3ptDerdzinski and W.\hskip2.3ptRoter, 
{\em The local structure of con\-for\-mal\-ly symmetric manifolds}, Bull.\
Belgian Math.\ Soc. \textbf{16} (2009), no.\,1, 117--128.

\bibitem{derdzinski-roter-10}A.\hskip2.3ptDerdzinski and
W\nnh.\hskip2.3ptRoter, 
{\em Compact pseu\-\hbox{do\hskip.7pt-}Riem\-ann\-i\-an manifolds with
parallel Weyl tensor}, Ann.\ Glob.\ Anal.\ Geom. \textbf{37} (2010), no.\,1,
73--90.

\bibitem{derdzinski-terek-cm}A.\hskip2.3ptDerdzinski and I.\hskip2.3ptTerek,
{\em Corrections of minor misstatements in several papers on ECS manifolds}, 
https:/\hskip-1pt/arxiv.org/pdf/2404.09766.pdf.

\bibitem{derdzinski-terek-tc}A.\hskip2.3ptDerdzinski and I.\hskip2.3ptTerek,
{\em The topology of compact rank-one ECS manifolds}, %to appear in 
Proc.\ Edinb.\ Math.\ Soc.\ (2) \textbf{66} (2023), no.\,3, 789--809.
%(preprint available 
%from https:/\hskip-1pt/arxiv.org/pdf/2210.09195.pdf).

\bibitem{derdzinski-terek-ro}A.\hskip2.3ptDerdzinski and I.\hskip2.3ptTerek,
{\em Rank-one ECS manifolds of di\-la\-tion\-al type}, %to appear in
Port.\ Math. \textbf{81} (2024), no.\,1, 69–96.

\bibitem{derdzinski-terek-ne}A.\hskip2.3ptDerdzinski and I.\hskip2.3ptTerek,
{\em New examples of compact Weyl-par\-al\-lel manifolds},
Monatsh.\ Math.\ \textbf{203} (2024), no. 4, 859–871.

\bibitem{derdzinski-terek-ms}A.\hskip2.3ptDerdzinski and I.\hskip2.3ptTerek, 
{\em The metric structure of compact rank-one ECS manifolds},
Ann.\ Global Anal.\ Geom. \textbf{64} (2023), no.\,4, Art.\ 24.
%preprint,%(available from %https:/\hskip-1pt/arxiv.org/pdf/2304.10388.pdf

\bibitem{deszcz-glogowska-hotlos-petrovic-torgasev-zafindratafa}R.\hskip2.3ptDeszcz,
M.\hskip2.3ptG{\l}ogowska, M.\hskip2.3ptHotlo\'s,
M.\hskip2.3ptPetrovi\'c-Torga\v sev and G.\hskip2.3ptZafindratafa, {\em
A note on some generalized curvature tensor}, Int.\ Electron.\ J.\ Geom. 
\textbf{16} (2023), no.\,1, 379--397.
%(preprint available from https:/\hskip-1pt/arxiv.org/pdf/2302.09387.pdf).

\bibitem{deszcz-glogowska-hotlos-zafindratafa}R.\hskip2.3ptDeszcz,
M.\hskip2.3ptG{\l}ogowska, M.\hskip2.3ptHotlo\'s and 
G.\hskip2.3ptZafindratafa, {\em On some curvature conditions of 
pseu\-do\-sym\-me\-try type}, Period.\ Math.\ Hungar. \textbf{70} (2015),
no.\,2, 153--170.

\bibitem{dundas}B.\hskip2.3ptI.\hskip2.3ptDundas, {\em A Short Course in 
Differential Topology}, Cambridge Mathematical Textbooks. Cambridge University 
Press, Cambridge, 2018.

\bibitem{hotlos}M.\hskip2.3ptHotlo\'s, {\em On con\-for\-mal\-ly symmetric
warped products}, Ann.\ Acad.\ Paedagog.\ Crac.\ Stud.\ Math. \textbf{4}
(2004), 75--85.

\bibitem{lang}S.\hskip2.3ptLang, {\em Differential and Riemannian Manifolds},
3rd ed., Graduate Texts in Mathematics \textbf{160}, Spring\-er-Ver\-lag, New 
York, 1995.

\bibitem{lee}J.\hskip2.3ptM.\hskip2.3ptLee, {\em Introduction to Topological  
Manifolds}, 2nd ed., Graduate Texts in Mathematics \textbf{202},
Spring\-er-Ver\-lag, New York, 2011.
  
\bibitem{mantica-suh}C.\hskip2.3ptA.\hskip2.3ptMantica and
Y\nnh.\hskip2.3ptJ.\hskip2.3ptSuh, {\em Con\-for\-mal\-ly  symmetric
manifolds and quasi con\-for\-mal\-ly recurrent Riemannian manifolds}, Balkan
J.\ Geom.\ Appl. \textbf{16} (2011), no.\,1, 66--77.

\bibitem{olszak}Z.\hskip2.3ptOlszak, {\em On con\-for\-mal\-ly recurrent
manifolds, I\hs{\rm:} Special distributions}, Zesz.\ Nauk.\ Po\-li\-tech.\ 
\'Sl., Mat.-Fiz. \textbf{68} (1993), 213--225.

\bibitem{roter}W\nnh.\hskip2.3ptRoter, {\em On con\-for\-mal\-ly symmetric 
Ric\-ci-re\-cur\-rent spaces}, Colloq.\ Math. \textbf{31} (1974), %no.\,1,
87--96.

\bibitem{schliebner}D.\hskip2.3ptSchliebner, {\em On the full holonomy of 
Lo\-rentz\-i\-an manifolds with parallel Weyl tensor}, preprint,
%(available from
https:/\hskip-1pt/arxiv.org/pdf/1204.5907.pdf%.).

\end{thebibliography}
\end{document}